\documentclass[11pt,a4paper]{article}

\usepackage{cite}
\usepackage[margin=1in]{geometry}
\usepackage{amsmath,amsthm,amssymb,amsfonts}

\newtheorem{proposition}{Proposition}
\newtheorem{theorem}{Theorem}
\newtheorem{corollary}{Corollary}
\newtheorem{lemma}{Lemma}
\DeclareMathOperator*{\argmin}{argmin}

\begin{document}
\date{}

\title{Network Consensus in the Wasserstein Metric Space of Probability Measures\thanks{This work was supported by AFOSR/AOARD via AOARD-144042. A preliminary draft of this work appeared in a conference proceedings as: ``A.N. Bishop and A. Doucet. Distributed nonlinear consensus in the space of probability measures. In Proc. of the 19th IFAC World Congress, Cape Town, South Africa, August 2014.''}}

\author{
  Adrian N. Bishop\\ CSIRO, and University of Technology Sydney\\
  \and Arnaud Doucet\\ University of Oxford\\
}

\maketitle

\begin{abstract}
Distributed consensus in the Wasserstein metric space of probability measures on the real line is introduced in this work. Convergence of each agent's measure to a common measure is proven under a weak network connectivity condition. The common measure reached at each agent is one minimizing a weighted sum of its Wasserstein distance to all initial agent measures. This measure is known as the Wasserstein barycenter. Special cases involving Gaussian measures, empirical measures, and time-invariant network topologies are considered, where convergence rates and average-consensus results are given. This work has possible applicability in computer vision, machine learning, clustering, and estimation.\end{abstract}

\section{Introduction}

The problem of distributed (network) consensus concerns a group of agents that seek to reach agreement upon certain state variables of interest by exchanging information across a network. Typically the agents are connected via a network that changes over time due to link failures, node failures, packet drops etc. For example, in distributed sensor networks the interaction topology may change over time as individual nodes (or some subset of such) may be mobile or unreliable or communication constraints are also present. All such variations in topology can happen randomly and often the network is disconnected for some time. Studies on the convergence of consensus algorithms (to a common agreed `value' at each agent) are often motivated by such complex time-varying networks.

\subsection{Background}

The consensus problem has a long history, e.g \cite{degroot1974reaching}, which is too broad to cover here. We highlight \cite{tsitsiklis:86,jadbabaie:03,saber:04,moreau:05,ren2005consensus,saber:07} for further history and background.

Many consensus algorithms have been proposed in the literature. References \cite{tsitsiklis:86,saber:04,ren2005consensus,cao2008reaching} focus on linear update rules (at each agent) and typically concern average-consensus or consensus about some linear function of all initial agent states in Euclidean space. The average-consensus problem has a natural relationship with distributed linear least squares or distributed (linear) maximum likelihood estimation \cite{xiao2005scheme} and distributed Kalman filtering \cite{spanos2005distributed,olfati2007distributed,carli2008distributed,cattivelli2010diffusion}. Alternative consensus algorithms using nonlinear update rules have been proposed and studied in \cite{moreau:05,hui2008distributed,yu2011consensus,ajorlou2011sufficient}. Here consensus to general functions (e.g. the maximum or minimum etc.) of all initial agent states may be sought as in \cite{cortes2008distributed,wang2010distributed} and even finite-time convergence may be achievable \cite{cortes2006finite,wang2010distributed}. One may also want to achieve consensus to some time-varying reference signal as in \cite{zhu2010discrete,hong2006tracking}.

We note here that the majority of the literature on consensus concerns agreement in Euclidean space as exemplified by the seminal papers of \cite{tsitsiklis:86,jadbabaie:03,saber:04,moreau:05,ren2005consensus,saber:07}.  However, there are exceptions. The problem of synchronisation is closely related to consensus but typically deals with the problem of driving a network of oscillators to a common frequency/phase. This work typically concerns nonlinear manifolds such as the circle.  A survey on synchronisation is given in \cite{strogatz2000kuramoto,dorfler2014synchronization} while consensus and synchronisation are related in \cite{li2010consensus}. Some other notable exceptions of consensus in non-Euclidean spaces are \cite{tuna2007consensus,sarlette2009consensus,Baras2010,sepulchre2011consensus,Grohs2012Wolfowitz,bonnabel2013stochastic}. In particular, \cite{sarlette2009consensus,sepulchre2011consensus} consider general nonlinear consensus on manifolds by embedding such manifolds in a suitable high-dimensional Euclidean space. In particular, this embedding approach is used to perform consensus on the special orthogonal group and on Grassmann manifolds. In \cite{bonnabel2013stochastic} a consensus algorithm on the Riemannian manifold of (Gaussian) covariance matrices is introduced under the Fisher metric (related to the Kullback-Leibler divergence). The authors in \cite{tuna2007consensus,Baras2010,Grohs2012Wolfowitz} study consensus in different metric spaces which is more closely related to the present work. For example, the author of \cite{Grohs2012Wolfowitz} develops an analogue of Wolfowitz's theorem \cite{jadbabaie:03} for a class of metric spaces with non-positive curvature which leads to a notion of consensus in such spaces.

The (distributed) consensus problem (as referenced above) has been widely studied across the fields of control and optimization, e.g. see again the seminal articles \cite{tsitsiklis:86,jadbabaie:03,saber:04,moreau:05,ren2005consensus,saber:07} published in the control literature, the references therein, and the many (thousands) of citing articles. This extensive and broad interest in distributed consensus protocols stems from interesting applications such as distributed estimation, filtering \cite{spanos2005distributed,olfati2007distributed,carli2008distributed,cattivelli2010diffusion} and distributed information fusion \cite{xiao2005scheme}, interesting applications in distributed control \cite{jadbabaie:03,saber:04,moreau:05,saber:07} and in distributed optimization \cite{sarlette2009consensus,bonnabel2013stochastic}, and many related topics too broad to discuss here, see also the survey articles \cite{garin2010survey,sepulchre2011consensus,dorfler2014synchronization}.

\subsection{Contributions}

The main contributions of this paper are a novel algorithm and convergence results for distributed consensus in the space of probability measures with time-varying interaction networks. We introduce a well-studied metric known as the Wasserstein distance which allows us to consider an important set of probability measures as a metric space \cite{givens1984class}. The proposed consensus algorithm is based on iteratively updating each agent's probability measure by finding a measure that minimizes the weighted sum of its Wasserstein distances to the agent's own previous measure plus all neighbour agents' measures. We show that convergence of the individual agents measures to a common probability measure is guaranteed under a weak network connectivity condition. The common measure that is achieved asymptotically at each agent is the one that is closest simultaneously to all initial agent measures in the sense of the Wasserstein distance. Focus in this work is on probability measures over the real line.

This work has potential applicability in the field of computer vision and image processing, distributed computation, clustering and data aggregation, distributed estimation, filtering and information fusion, distributed optimization in metric spaces, and machine learning more broadly, among other fields discussed later. Applications and related topics, particularly relevant in the Wasserstein domain, are discussed later. However, the main focus of this article is on the general Wasserstein consensus idea itself and its convergence.

This paper extends \cite{bishopdoucetIFAC} with the addition of results detailing convergence rates and network properties in which particular consensus values may be achieved.

\subsection{Paper Organization}

The main contribution is given in Section \ref{consensusWasserstein} where a consensus algorithm in the space of probability measures is introduced, and its convergence is studied. In Section \ref{specialcases} we establish for specific scenarios - initial Gaussian measures; for empirical measures; and time-invariant networks - the exponential convergence of this consensus algorithm and various computational aspects. In Section \ref{applications} we discuss potential applications. Concluding remarks are given in Section \ref{conclusion}.

\subsection{Notation and Conventions}

Consider a group of agents indexed in $\mathcal{V}=\{1,\ldots,n\}$ and a set of (possibly) time-varying undirected links $\mathcal{E}(t)\subset\mathcal{V}\times\mathcal{V}$ defining a network graph $\mathcal{G}(t)(\mathcal{V},\mathcal{E}(t))$. The neighbour set at agent $i$ is denoted by $\mathcal{N}_{i}(t)=\{j\in\mathcal{V}:(i,j)\in\mathcal{E}(t)\}$. Time is indexed using $\mathbb{N}$.

The graph adjacency matrix $\mathbf{A}(t)\in\mathbb{R}^{n\times n}$ obeys $\mathbf{A}\left(t\right)=\mathbf{A}\left(t\right)^{\top}=[a_{ij}(t)]$ where $a_{ij}(t)=1\Leftrightarrow(i,j)\in\mathcal{E}(t)$ and $a_{ij}(t)=0$ otherwise. Implicit throughout is that $a_{ii}(t)=1$ for all $i$ and $t$ and thus $i\in\mathcal{N}_{i}(t)$ for all $t$. A weighted adjacency matrix is denoted by $\mathbf{W}(t)=[w_{ij}(t)]\in\mathbb{R}^{n\times n}$ with $a_{ij}(t)=1\Leftrightarrow w_{ij}(t)>0$ and $w_{ij}=0$ otherwise. We require $\sum_{j\in\mathcal{N}_{i}(t)}w_{ij}(t)=1$ and assume that $w_{ii}(t)>0$ for all $i$ and $t$ so that $w_{ij}(t)\in[0,1)$ whenever $i\neq j$ and for all $t$.

The adjacency matrix $\mathbf{A}\left(t\right)$ defines $\mathcal{G}(t)(\mathcal{V},\mathcal{E}(t))$ and vice versa because $a_{ij}(t)=1\Leftrightarrow(i,j)\in\mathcal{E}(t)$ and $a_{ij}(t)=0\Leftrightarrow(i,j)\notin\mathcal{E}(t)$. The weighted adjacency matrix defines $\mathcal{G}(t)(\mathcal{V},\mathcal{E}(t))$ because $w_{ij}(t)>0\Leftrightarrow(i,j)\in\mathcal{E}(t)$ and $w_{ij}(t)=0\Leftrightarrow(i,j)\notin\mathcal{E}(t)$, but $\mathcal{G}(t)(\mathcal{V},\mathcal{E}(t))$ alone defines only the sparsity pattern of $\mathbf{W}(t)$.

Consider the sequence of graphs $\mathcal{G}(t_{k}),\mathcal{G}(t_{k+1}),\ldots,\mathcal{G}(t_{k+T})$ on the same vertex set $\mathcal{V}$. The union of this sequence is denoted by $\mathfrak{G}(t_{k},t_{k+T})(\mathcal{V},\cup_{t\in[t_{k},t_{k+T}]}\mathcal{E}(t))$, i.e. $\mathfrak{G}(t_{k},t_{k+T})$ is just a graph on the vertex set $\mathcal{V}$ with edges $\cup_{t\in[t_{k},t_{k+T}]}\mathcal{E}(t)$. The sequence is said to be jointly connected if $\mathfrak{G}$ is connected.

\section{Consensus in the Wasserstein Space of Probability Measures}
\label{consensusWasserstein}

The main contribution of this work is given in this section where we introduce and establish the convergence of a consensus algorithm in the Wasserstein metric space of probability measures.

Suppose the state of agent $i$ is given by a Radon probability measure $\mu_{i}$ defined on the Borel sets of $(\mathbb{R},d)$ where in this section we restrict $d:\mathbb{R}\times\mathbb{R}\rightarrow\lbrack0,\infty)$ to be the usual Euclidean distance. Define the space of all such measures on $(\mathbb{R},d)$ by $\mathfrak{U}(\mathbb{R})$ and the subset of all such measures with finite $p^{th}$ moment by $\mathfrak{U}_{p}(\mathbb{R})$ where henceforth we assume that $2\leq p<\infty$. That is, $\mathfrak{U}_{p}$ is the collection of probability measures such that $\int_{\mathbb{R}}d({x},{x}_{0})^{p}\,\mathrm{d}\mu_{i}({x})<\infty$ for a given, arbitrary, ${x}_{0}\in\mathbb{R}$.

One can associate the Wasserstein metric $\ell_{p}:\mathfrak{U}_{p}(\mathbb{R})\times\mathfrak{U}_{p}(\mathbb{R})\rightarrow[0,\infty)$ with $\mathfrak{U}_{p}$ which is defined by
\[
	\ell_{p}(\mu_{i},\mu_{j})=\left(\inf_{\gamma\in\Gamma(\mu_{i},\mu_{j})}\int_{\mathbb{R}\times\mathbb{R}}d(x_{i},x_{j})^{p}\,\mathrm{d}\gamma(x_{i},x_{j})\,\right)^{1/p}
\]

\noindent where $\Gamma(\mu_{i},\mu_{j})$ denotes the collection of all probability measures on $\mathbb{R}\times\mathbb{R}$ with marginals $\mu_{i}$ and $\mu_{j}$ on the first and second factors; see \cite{cv:03a,ambrosio2005gradient}.

Let us recall some standard results about the Wasserstein metric space $(\mathfrak{U}_{p}(\mathbb{R}),\ell_{p})$ when $p\geq2$; see e.g. \cite{cv:03a,ambrosio2005gradient,Kloeckner2010,bertrand2012geometric}.

\begin{enumerate}
\item $(\mathfrak{U}_{p}(\mathbb{R}),\ell_{p})$ is a complete and separable metric space.
\item $\lim_{k\rightarrow\infty}\ell_{p}(\mu_{k},\mu)=0$ is equivalent to weak convergence and convergence of the first $p$ moments.
\item Given two measures $\mu_{i},\mu_{j}\in\mathfrak{U}_{p}(\mathbb{R})$ then $\ell_{p}(\mu_{i},\mu_{j})=\ell_{p}(\mu_{i},\mu)+\ell_{p}(\mu_{j},\mu)$ for some $\mu\in\mathfrak{U}_{p}(\mathbb{R})$.
\item More generally, there exists a continuously parameterised constant speed path $\mu_{s}\in\mathfrak{U}_{p}(\mathbb{R})$, $s\in[0,1]$ such that for $\mu_{i},\mu_{j}\in\mathfrak{U}_{p}$ we have $\mu_{s=0}=\mu_{i}$ and $\mu_{s=1}=\mu_{j}$ and $\ell_{p}(\mu_{i},\mu_{j})=\ell_{p}(\mu_{i},\mu_{s})+\ell_{p}(\mu_{j},\mu_{s})$, $\forall s\in[0,1]$. The measure $\mu_{s}$ is known as the interpolant measure \cite{mccann1997convexity}.
\item The interpolant measure defines a geodesic and consequently $(\mathfrak{U}_{p},\ell_{p})$ is geodesic.
\item $(\mathfrak{U}_{p}(\mathbb{R}),\ell_{p})$ has vanishing curvature in the sense of Alexandrov (a subset of CAT(0)); see Proposition 4.1 in \cite{Kloeckner2010}.%
\item $(\mathfrak{U}_{p}(\mathbb{R}),\ell_{p})$ is simply connected; see \cite{Kloeckner2010,bertrand2012geometric}.
\end{enumerate}

All metrics are continuous and we recall that a constant speed geodesic in $(\mathfrak{U}_{p}(\mathbb{R}),\ell_{p})$ is a curve $\mu_{s}:\mathbb{I}\rightarrow\mathfrak{U}_{p}$ parameterised on some interval $\mathbb{I}\subset\mathbb{R}$ that satisfies $\ell_{p}(\mu_{s_{i}},\mu_{s_{j}})=v|s_{i}-s_{j}|$ for some constant $v>0$ and for all $s_{i},s_{j}\in\mathbb{I}$.\medskip

Suppose the measure at agent $i$ is updated by
\begin{equation}
\boxed{\mu_{i}(t+1)~=~\argmin_{\eta\in\mathfrak{U}_{p}(\mathbb{R})}~\sum_{j\in\mathcal{N}_{i}(t)}w_{ij}(t)\,\ell_{p}(\eta,\mu_{j}(t))^{p}\label{consensus3}}
\end{equation}
for all $i\in\mathcal{V}$ where we recall that we assume $\sum_{j\in\mathcal{N}_{i}(t)}w_{ij}(t)=1$ and $w_{ii}(t)>0$ so that consequently $w_{ij}(t)\in[0,1)$ whenever $i\neq j$. This operation is well-defined as discussed below.

Application of the update rule (\ref{consensus3}) to each agent $i\in\mathcal{V}$ corresponds to the proposed nonlinear (distributed) consensus algorithm. Note that we consider only undirected (or bidirectional) network communications in this work for simplicity. Directed communication may be studied as in, e.g., \cite{moreau:05,ren2005consensus}, with additional conditions needed for convergence in that setting \cite{moreau:05}.

\subsection{Main Result}

We state here our main result.

\begin{theorem} Consider a group of agents $\mathcal{V}$ and network $\mathcal{G}(t)(\mathcal{V},\mathcal{E}(t))$ where each agent $i$ has initial state $\mu_{i}(0)\in\mathfrak{U}_{p}\left(\mathbb{R}\right)$ and updates its state $\mu_{i}(t)\in\mathfrak{U}_{p}(\mathbb{R})$ according to (\ref{consensus3}). If for all $t_{0}\in\mathbb{N}$ the graph union $\mathfrak{G}(t_{0},\infty)$ is connected then there exists $\mu^{\ast}\in\mathfrak{U}_{p}(\mathbb{R})$ such that
\[
	\lim_{t\rightarrow\infty}\ell_{p}(\mu_{i}\left(t\right),\mu^{\ast})=0
\]
for any $i\in\mathcal{V}$.
\end{theorem}

The proof of Theorem 1 is given in the next subsection following the provision of a number of supporting results.

\subsection{Proof of the Main Result}

Before proceeding to the general proof, we note briefly here that the Wasserstein metric $\ell_{p}:\mathfrak{U}_{p}(\mathbb{R})\times\mathfrak{U}_{p}(\mathbb{R})\rightarrow[0,\infty)$ may be written as,
\[
	\ell_{p}(\mu_{i},\mu_{j})=\left(\,\int_0^1~|F_{i}^{-}(x) - F_{j}^{-}(x) |^p \,dx \,\right)^{1/p}
\]
where $F_{i}^{-}(x):[0,1]\rightarrow\mathbb{R}$ is the inverse cumulative distribution function for $\mu_{i}(t)\in\mathfrak{U}_{p}(\mathbb{R})$, defined in more detailed later. This form for $\ell_{p}(\mu_{i},\mu_{j})$ is given by P. Major in \cite{major1978invariance}. It follows with $p=2$ that the solution to (\ref{consensus3}) at any $i\in\mathcal{V}$ and any $t\in\mathbb{N}$ has an inverse cumulative distribution function given by
\[
	F_{i}^{-}(t+1)(x)={\textstyle \sum_{j\in\mathcal{N}_{i}}w_{ij}F_{j}^{-}(t)(x)}
\]
for all $x\in[0,1]$, as proven more formally later. From this formulation one can derive a kind of convergence result for consensus with the update rule (\ref{consensus3}); e.g. see Proposition \ref{convergedSpeedProp} later, its proof, and e.g. equation (\ref{invcumupdate}) later and the surrounding discussion there. This formulation of the metric in terms of the inverse cumulative distribution functions conveys intuition on the space $(\mathfrak{U}_{p}(\mathbb{R}),\ell_{p})$, for example implying immediately that $(\mathfrak{U}_{p}(\mathbb{R}),\ell_{p})$ is isometric to a convex subset of the Banach space $(L_{p}([0,1]),d)$ and thus inherits many familiar properties (which is rather unique to the case of measures on $\mathbb{R}$).

For generality we proceed with the proof using more metric space/geometric arguments that in many places may carry over to the more general case involving measures on $\mathbb{R}^n$ (albeit we do not explore those cases here). Where ideas do not carry over we highlight that in this more general geometric formalism so as to distinguish where a generalised proof may need to be extended.

The proof proceeds now via a series of supporting lemmas. Note that a subset $\mathfrak{X}\subset\mathfrak{U}_{p}(\mathbb{R})$ is convex if every geodesic segment whose endpoints are in $\mathfrak{X}$ lies entirely in $\mathfrak{X}$. The (closed) convex hull $\mathrm{co}(\mathfrak{Y})$ of a subset $\mathfrak{Y}\subset\mathfrak{U}_{p}$ is the intersection of all (closed) convex subsets of $\mathfrak{U}_{p}$ that contain $\mathfrak{Y}$.

\begin{lemma}
If $\mu_{i}(t)\in\mathfrak{U}_{p}(\mathbb{R})$ then the operation (\ref{consensus3}) is well-defined in the sense that it admits a solution and this solution is unique whenever (at least) one $\mu_{j}(t)\in\mathfrak{U}_{p}$, $j\in\mathcal{N}_{i}(t)$ does not give support to small sets\footnote{A small set is defined \cite{agueh2011barycenters} as a set of Hausdorff dimension $0$. This condition plays a role only in uniqueness and it is generally unnecessary \cite{agueh2011barycenters,bigot2012consistent}. However, this requirement does exclude empirical measures on $\mathbb{R}$ which arise in numerous applications relevant to this work (as discussed later). Luckily, it is generically true (i.e. excluding particular, non-generic, arrangements) that (\ref{consensus3}) has a unique solution even in such cases; see \cite{bigot2012consistent,bonneel2015sliced}.
Note if all inputs are discrete we allow for both common and uncommon supports. Going forward we will not repeatedly call on the need for (at least) one initial measure to exclude support on small sets and later results may be read as implicitly assuming uniqueness (or implicitly assuming exclusion of sole support on small sets).}.
\end{lemma}

Recall that $(\mathfrak{U}_{p}(\mathbb{R}),\ell_{p})$ is CAT(0) (indeed it has the stronger property of vanishing curvature, e.g. owing to its relationship with the Banach space $(L_{p}([0,1]),d)$),  in addition to being uniquely geodesic, complete and separable, i.e. $(\mathfrak{U}_{p}(\mathbb{R}),\ell_{p})$ is a Hadamard space. This lemma then follows from the fact that $(\mathfrak{U}_{p}(\mathbb{R}),\ell_{p})$ is Hadamard and Fr\'{e}chet averages such as defined by operations of the form (\ref{consensus3}) are well defined in such spaces; see page 334 in \cite{burago2001course}. The existence and uniqueness of solutions to (\ref{consensus3}) is also discussed in \cite{agueh2011barycenters} more generally; see also \cite{bigot2012consistent,bonneel2015sliced}. It is worth noting in passing the related work in \cite{tuna2007consensus,Grohs2012Wolfowitz} which deals with similar consensus topics in CAT(0) spaces, and \cite{Baras2010} which deals with consensus in a general class of convex metric spaces.

The convex hull of the set of measures $\{\mu_{i}\}$, $i\in\widetilde{\mathcal{V}}\subseteq\mathcal{V}$, is defined by
\[
	\mathrm{co}(\{\mu_{i}\})=\{\argmin_{\eta\in\mathfrak{U}_{p}(\mathbb{R})}\sum_{i\in\widetilde{\mathcal{V}}}w_{i}\ell_{p}(\eta,\mu_{i})^{p}|w_{i}\geq0,\sum_{i}w_{i}=1\}.
\]

\begin{lemma} Consider a collection $\{\mu_{i}\}$, $i\in\widetilde{\mathcal{V}}\subseteq\mathcal{V}$ of distinct measures in $(\mathfrak{U}_{p}(\mathbb{R}),\ell_{p})$. The convex hull of $\{\mu_{i}\}$ is $\mathrm{co}(\{\mu_{i}\})\subset\mathfrak{U}_{p}(\mathbb{R})$ and is isometric to a $l$-sided convex polygon in $\mathbb{R}^{2}$ with $2\leq l\leq|\{\mu_{i}\}|$.
\end{lemma}

Before proceeding with the proof we point to \cite{bridson1999metric} for background on comparison triangles and Alexandrov curvature of metric spaces. We also note that in a general geodesic CAT(0) space, i.e. some arbitrary geodesic space with non-positive curvature, the preceding lemma is not true and the convex hull of a `geodesic triangle' defined by three points in such spaces may be of dimension greater than two\footnote{Our Euclidean intuition is generally wrong when it suggests the existence of a two-dimensional convex hull for a triangle defined by three points and the geodesics connecting them (albeit this is hard to visualise of course).}; see Chapter II.2 in \cite{bridson1999metric}.

\begin{proof}
Lemma 2 is a simple consequence of the vanishing curvature property of $(\mathfrak{U}_{p}(\mathbb{R}),\ell_{p})$. We elaborate for completeness. $(\mathfrak{U}_{p}(\mathbb{R}),\ell_{p})$ has vanishing curvature in the sense of Alexandrov (see Proposition 4.1 in \cite{Kloeckner2010}) which formally means that for any triangle of points $\{\mu_{i}\}$, $i\in\{i_{1},i_{2},i_{3}\}$ and any point on the geodesic $\mu_{s}\in\mathfrak{U}_{p}(\mathbb{R})$, $s\in[0,1]$ such that, for example, $\mu_{s=0}=\mu_{i_{1}}$ and $\mu_{s=1}=\mu_{i_{2}}$ then the $\ell_{p}$ distance between $\mu_{i_{3}}$ and $\mu_{s}$, $s\in[0,1]$ is the same as the corresponding Euclidean distance in a comparison triangle in $\mathbb{R}^{2}$. Consider also any pair of points $\mu_{j}$ and $\mu_{k}$ with $\mu_{j}$ on the geodesic connecting $\mu_{i_{1}}$ and $\mu_{i_{2}}$ and $\mu_{k}$ on the geodesic connecting $\mu_{i_{1}}$ and $\mu_{i_{3}}$ with $\{\mu_{j},\mu_{k}\}\cap\{\mu_{i}\}=\emptyset$, $i\in\{i_{1},i_{2},i_{3}\}$. Then vanishing curvature also implies $\angle_{\mu_{i_{1}}}(\mu_{j},\mu_{k})$ is equal to the usual interior Euclidean angle at the corresponding vertex in the comparison triangle in $\mathbb{R}^{2}$. Here the angle $\angle_{\mu_{i_{1}}}(\mu_{j},\mu_{k})$ is the Alexandrov angle in arbitrary metric spaces; see Chapter II.1 in \cite{bridson1999metric}. It now follows that the convex hull of any triangle of points $\{\mu_{i}\}$
in $(\mathfrak{U}_{p}(\mathbb{R}),\ell_{p})$ is isometric to a triangle in $\mathbb{R}^{2}$; e.g. see Proposition 2.9 (Flat Triangle Lemma) in \cite{bridson1999metric}. Now define $\mathfrak{C}=\{\Delta_{j}\}$ to be the collection of geodesic triangles in $(\mathfrak{U}_{p}(\mathbb{R}),\ell_{p})$ defined by every combination of three points in $\{\mu_{i}\}$, $i\in\widetilde{\mathcal{V}}\subseteq\mathcal{V}$. Clearly $\mathrm{co}(\{\mu_{i}\})=\cup_{j}\Delta_{j}$. Consider also the corresponding collection $\mathfrak{C}^{+}=\{\Delta_{j}^{+}\}$ of comparison triangles in $\mathbb{R}^{2}$. The Flat Triangle Lemma implies that this collection can be arranged in $\mathbb{R}^{2}$ such that each angle $\angle_{\mu_{i}}(\mu_{j},\mu_{k})$ and each distance $\ell_{p}(\mu_{i},\mu_{j})$ for all $i,j,k\in\widetilde{\mathcal{V}}$ in $(\mathfrak{U}_{p}(\mathbb{R}),\ell_{p})$ equals exactly the corresponding angle or distance in the comparison configuration of points in $\mathbb{R}^{2}$. Obviously, the convex hull of the comparison configuration is a $l$-sided convex polygon in $\mathbb{R}^{2}$ with $2\leq l\leq|\{\mu_{i}\}|$ and equal to $\cup_{j}\Delta_{j}^{+}$. Define the following map
\begin{equation}
	f_{\ell_{p},d}:\mathrm{co}(\{\mu_{i}\})\rightarrow\mathbb{R}^{2},~~i\in\widetilde{\mathcal{V}}\label{isometrytoR}
\end{equation}
so the restriction
\begin{eqnarray*}
	f_{\ell_{p},d}(\Delta_{j}) & = & f_{\ell_{p},d}(\mathrm{co}(\{\mu_{j_{1}},\mu_{j_{2}},\mu_{j_{3}}\}))\\
 		& = & \mathrm{co}(\{f_{\ell_{p},d}(\mu_{j_{1}}),f_{\ell_{p},d}(\mu_{j_{2}}),f_{\ell_{p},d}(\mu_{j_{3}})\})=\Delta_{j}^{+}
\end{eqnarray*}
$\forall j\in\mathfrak{C}=\{\Delta_{j}\}$ is an isometry. Then
\[
	f_{\ell_{p},d}(\mathrm{co}(\{\mu_{i}\}))=f_{\ell_{p},d}(\cup_{j}\Delta_{j})=\cup_{j}f_{\ell_{p},d}(\Delta_{j})=\cup_{j}\Delta_{j}^{+}
\]
from the Flat Triangle Lemma and the property of vanishing curvature. For any two points in $\mathrm{co}(\{\mu_{i}\})$ there exists a $\Delta_{j}\in\mathfrak{C}$ that contains them and the restriction $f_{\ell_{p},d}(\Delta_{j})$ is an isometry to a convex subset of $\cup_{j}\Delta_{j}^{+}$. Thus, $f_{\ell_{p},d}$ is an isometry and this completes the proof.
\end{proof}

\begin{lemma}
Consider the convex hull $\mathrm{co}(\{\mu_{j}(t)\})$, with $j\in\mathcal{N}_{i}(t)$ at time $t$. If agent $i$ applies (\ref{consensus3}) it follows that $\mu_{i}(t+1)$ is strictly within the convex hull $\mathrm{co}(\{\mu_{j}(t)\})$ whenever $|\{\mu_{j}(t)\}|\geq2$ and two agent states are distinct and $w_{ij}(t)\in(0,1)$.
\end{lemma}

\begin{proof}
It is enough to consider two agents \footnote{This is because Fr\'{e}chet averages in Euclidean space (on a set of input points) are associative, and can be found iteratively by computing the average initially for a pair of points (in a larger set of inputs), and then computing the average between this result and the next point (in the input set), and so on (adjusting the weights defining the average each time), see \cite{boissard2011distribution}. Owing to Lemma 2, this associativity property still holds here.} $i,j\in\mathcal{V}$ with (\ref{consensus3}) then given by
\begin{eqnarray*}
	\mu_{i}(t+1) & = & \argmin_{\eta\in\mathfrak{U}_{p}(\mathbb{R})}~w_{ii}(t)\left(\ell_{p}(\eta,\mu_{i}(t))^{p}-\ell_{p}(\eta,\mu_{j}(t))^{p}\right)+~\ell_{p}(\eta,\mu_{j}(t))^{p}
\end{eqnarray*}
and to note that $\eta$ must lie on a geodesic $\mu_{s}:\mathbb{I}\rightarrow\mathfrak{U}_{p}(\mathbb{R})$. The proof relies on showing that $\mu_{i}(t+1)\notin\{\mu_{i}(t),\mu_{j}(t)\}$ when $w_{ii},w_{ij}\in(0,1)$. The first term
\[
	w_{ii}(t)\left(\ell_{p}(\eta,\mu_{i}(t))^{p}-\ell_{p}(\eta,\mu_{j}(t))^{p}\right)
\]
is strictly negative at $\eta=\mu_{i}(t)$ and strictly increasing as $\eta$ moves from $\mu_{i}(t)$ to $\mu_{j}(t)$ and conversely $\ell_{p}(\eta,\mu_{j}(t))^{p}$ is strictly positive at $\eta=\mu_{i}(t)$ and strictly decreasing to zero as $\eta$ moves from $\mu_{i}(t)$ to $\mu_{j}(t)$. Then for any $w_{ii}\in(0,1)$ and because $\ell_{p}$ is continuous it follows that there exists some $\mu_{\varepsilon}$ on $\mu_{s}$ with $\varepsilon>0$ such that
\begin{eqnarray*}
w_{ii}(t)\left(\ell_{p}(\eta,\mu_{i}(t))^{p}-\ell_{p}(\eta,\mu_{j}(t))^{p}\right) & < & 0\\
|w_{ii}(t)\left(\ell_{p}(\eta,\mu_{i}(t))^{p}-\ell_{p}(\eta,\mu_{j}(t))^{p}\right)| & < & \ell_{p}(\eta,\mu_{j}(t))^{p}
\end{eqnarray*}
on $\eta\in\mu_{s}$, $s\in[0,\varepsilon]$. Consequently, $\mu_{i}(t+1)$ is strictly decreasing on $\eta\in\mu_{s}$, $s\in[0,\varepsilon]$. Hence for any $w_{i1}\in(0,1)$ the point $\mu_{i}(t)$ cannot be a minimum. The same argument applies to $\mu_{j}(t)$.
\end{proof}

The following is a simple consequence of the preceding result.

\begin{corollary}
Consider the convex hull $\mathrm{co}(\{\mu_{i}(0)\})$ of all initial agent states in $(\mathfrak{U}_{p}(\mathbb{R}),\ell_{p})$. If each agent applies (\ref{consensus3}) it follows that $\mathrm{co}(\{\mu_{i}(t)\})\subseteq\mathrm{co}(\{\mu_{i}(0)\})$ for all $t$.
\end{corollary}

The next result concerns an importance special case of the main result.

\begin{lemma} Suppose $\mathcal{G}(\mathcal{V},\mathcal{E})$ is time-invariant and connected. Suppose the state of each agent is $\mu_{i}(t)\in\mathfrak{U}_{p}(\mathbb{R})$ and that each agent applies (\ref{consensus3}). Then there exists $\mu^{\ast}$$\in\mathfrak{U}_{p}(\mathbb{R})$ such that for any $i\in\mathcal{V}$ it holds that
\[
	\lim_{t\rightarrow\infty}\ell_{p}(\mu_{i}\left(t\right),\mu^{\ast})=0.
\]
\end{lemma}

\begin{proof}
It almost goes without saying that $\ell_{2}(\mu_{i}(t),\mu_{j}(t))^{p}=0$, $\forall i,j\in\mathcal{V}$ with a constant $\mu_{i}(t)$ in $\mathfrak{U}_{p}(\mathbb{R})$ is an equilibrium state of (\ref{consensus3}). Consider a Lyapunov-like function $\nu(\mu):\mathfrak{U}_{p}^n(\mathbb{R})\rightarrow\mathbb{R}$ given by
\begin{equation}
	\nu(\mu)=\sup_{\eta,\chi\in\{\mu_{i}(t)\}_{i\in\mathcal{V}}}\ell_{p}(\eta,\chi)^{p}\label{lyapunovcorollary}
\end{equation}

\noindent and note that $\nu(\mu)\geq0$ with $\nu(\mu)=0$ if and only if $\mu_{i}=\mu_{j}$ for all $i,j\in\mathcal{V}$ \footnote{We abuse notation here slightly. We use the shorthand $\mu$ as the argument in $\nu(\mu)$ to represent the collection of all agent measures $\{\mu_{i}(t)\}_{i\in\mathcal{V}}$ with $|\mathcal{V}|=n$.}. By Corollary 1 it follows that $\nu(\mu)$ is non-increasing along trajectories of (\ref{consensus3}). It suffices to show $\nu(\mu(t+n-1))<\nu(\mu(t))$ for each $t$. Firstly, pick a $t_{0}\geq0$ and note $\mathrm{co}(\{\mu_{i}(t_{0})\})\subseteq\mathrm{co}(\{\mu_{i}(0)\})$ and $f_{\ell_{p},d}(\mathrm{co}(\{\mu_{i}(t_{0})\}))\subseteq f_{\ell_{p},d}(\mathrm{co}(\{\mu_{i}(0)\}))$ from Corollary 1 and where $f_{\ell_{p},d}$ is an isometry given by (\ref{isometrytoR}). Without loss of generality, via Lemma 2, suppose that $f_{\ell_{p},d}(\mathrm{co}(\{\mu_{i}(t_{0})\}))$ is a $l$-sided polygon in $\mathbb{R}^{2}$ with $2\leq l\leq|\mathcal{V}|$ on the collection of vertices $\{\mathbf{x}_{j}(t_{0})\}$, $j\in\{1,\ldots,l\}$ with $\mathbf{x}_{j}(t_{0})\in\mathbb{R}^{2}$. If we chose a $t_{0}$ such that $l=1$ then we would be done. Define the following set-valued function
\begin{equation}
	h_{j}(t)=\left\{ i\in\mathcal{V}:f_{\ell_{p},d}(\mu_{i}(t))=\mathbf{x}_{j}(t_{0})\right\} ,~~\forall j\in\{1,\ldots,l\}\label{setpolygonvertices}
\end{equation}
for each time $t\geq t_{0}$. It is immediate from Lemma 3 that $h_{j}(t+1)\subseteq h_{j}(t)$ for all $j\in\{1,\ldots,l\}$; i.e. more generally, no agent state $f_{\ell_{p},d}(\mu_{i}(t))$ which is not on the boundary of the $l$-sided polygon at time $t$ can ever reach this same boundary at $t+1$ as a consequence of Lemma 3. Note that $|h_{j}(t_{0})|\leq n-1$ for all $j\in\{1,\ldots,l\}$ with $l\geq2$ at $t_{0}$. Recall the neighbour set at agent $i$ is given by $\mathcal{N}_{i}(t)$. Because the network is connected, for each $k\in h_{j}(t_{0})$ the neighbour set obeys $\mathcal{N}_{k}(t_{0})\neq\emptyset$ for each $j\in\{1,\ldots,l\}$. Then by Lemma 3 it follows that $h_{j}(t_{0}+1)\subset h_{j}(t_{0})$ since at least one $k\in h_{j}(t_{0})$ must be connected to an agent outside $h_{j}(t_{0})$ and this agent's state must change $\mu_{k}(t_{0})\neq\mu_{k}(t_{0}+1)$ as a consequence of Lemma 3 such that $f_{\ell_{p},d}(\mu_{k}(t_{0}+1))\neq\mathbf{x}_{j}$. At the next time $t_{0}+1$ it holds again that for each $k\in h_{j}(t_{0}+1)$ (assuming $h_{j}(t_{0}+1)\neq\emptyset$) the neighbour set obeys $\mathcal{N}_{k}(t_{0}+1)\neq\emptyset$ for each $j\in\{1,\ldots,l\}$. Then by application of Lemma 3 it follows again that $h_{j}(t_{0}+2)\subset h_{j}(t_{0}+1)\subset h_{j}(t_{0})$. Thus, $h_{j}(t+1)\subset h_{j}(t)$ is a strictly decreasing set-valued function unless $h_{j}(t)=\emptyset$. By at most time $t_{0}+n-1$ it follows that $h_{j}(t_{0}+n-1)=\emptyset$ and the argument can reset by redefining $t_{0}$. It follows that $f_{\ell_{p},d}(\mathrm{co}(\{\mu_{i}(t_{0}+n-1)\}))\subset f_{\ell_{p},d}(\mathrm{co}(\{\mu_{i}(t_{0})\}))$ for all $t_{0}\geq0$. Following the proof of Lemma 2 we know $\mathrm{co}(\{\mu_{i}(t_{0}+n-1)\})\subset\mathrm{co}(\{\mu_{i}(t_{0})\})$ and thus because we chose $t_{0}$ arbitrarily $\nu(\mu(t+n-1))<\nu(\mu(t))$ for each $t\in\mathbb{N}$ unless $\mu_{i}(t+n-1)=\mu_{i}(t)$, $\forall i$, as desired. The existence of a strictly decreasing Lyapunov function completes the proof.
\end{proof}

The preceding lemma specialises this theorem to the case where the network topology is connected and time-invariant (but otherwise arbitrary). This lemma is of interest on its own in many applications in which the topology is static. Proof of this lemma, given Lemmas 1-3, follows roughly the analysis of \cite{moreau:05} on nonlinear consensus in the usual Euclidean metric space.

We are now ready to prove Theorem 1.

\begin{proof}(\textbf{of Theorem 1})
The proof here relies on extending the previous lemma to the case where $\mathcal{G}(t)(\mathcal{V},\mathcal{E}(t))$ is time-varying and for all $t_{0}\in\mathbb{N}$ the graph union $\mathfrak{G}(t_{0},\infty)$ is connected. Recall the same Lyapunov function (\ref{lyapunovcorollary}) as used in the proof of Lemma 4 (we assume familiarity with the proof of Lemma 4 going forward).

We note that it suffices to show that there is a countably infinite number of finite time intervals $t\in[t_{0}^{q},\widehat{t}_{0}^{q}]$, $q\in\mathbb{N}$ such that $\nu(\mu(t_{0}^{q}+\widehat{t}_{0}^{q}))<\nu(\mu(t_{0}^{q}))$.

Pick $t_{0}^{q}\geq0$, $q\in\mathbb{N}$ so $f_{\ell_{p},d}(\mathrm{co}(\{\mu_{i}(t_{0}^{q})\}))$ is a $l$-sided polygon in $\mathbb{R}^{2}$ with $2\leq l\leq|\mathcal{V}|$ on the collection of vertices $\{\mathbf{x}_{j}(t_{0}^{q})\}$, $j\in\{1,\ldots,l\}$ with $\mathbf{x}_{j}(t_{0}^{q})\in\mathbb{R}^{2}$. Recall (\ref{setpolygonvertices}). Then define a sequence of times $\{t_{s(j)}^{q}\}$, $s(j)\in\mathbb{N}$ each greater than $t_{0}^{q}$ for each $j\in\{1,\ldots,l\}$ with $l\geq2$. The connectivity condition implies the existence of such a sequence for each $j$ with the property that, if $h_{j}(t_{s(j)}^{q})\neq\emptyset$, there exists a $k\in h_{j}(t_{s(j)}^{q})$ that is connected to an agent outside $h_{j}(t_{s(j)}^{q})$. Then, this agent's state must change $\mu_{k}(t_{s(j)}^{q})\neq\mu_{k}(t_{s(j)}^{q}+1)$ as a consequence of Lemma 3 and $f_{\ell_{p},d}(\mu_{k}(t_{s(j)}^{q}+1))\neq\mathbf{x}_{j}(t_{0}^{q})$. Then $h_{j}(t_{s(j)}^{q}+1)\subset h_{j}(t_{s(j)}^{q})$ for all $j\in\{1,\ldots,l\}$ unless obviously $h_{j}(t_{s(j)}^{q})=\emptyset$. As in the proof of Lemma 4 it holds that $s(j)\geq n-1$ implies $h_{j}(t_{s(j)+1}^{q})=\emptyset$ for all $j$. Let $\widehat{t}_{0}^{q}=\min\{t\in\mathbb{N}:t>t_{0}^{q},~s(j)\geq n-1,\forall j\}$ and note then that the interval $t\in[t_{0}^{q},\widehat{t}_{0}^{q}]$ is finite owing to the connectivity condition. Moreover, as in the proof of Lemma 4 one can then show that $\nu(\mu(\widehat{t}_{0}^{q}))<\nu(\mu(t_{0}^{q}))$. Restart the argument by picking $t_{0}^{q+1}$ to be equal or sufficiently close to $\widehat{t}_{0}^{q}$ and note that the connectivity condition then implies the number of such (finite) intervals $t\in[t_{0}^{q},\widehat{t}_{0}^{q}]$ is countably infinite on $q\in\mathbb{N}$.

We thus have a strictly decreasing Lyapunov function $\nu(\mu(\widehat{t}_{0}^{q}))<\nu(\mu(t_{0}^{q}))$ on the sequence of finite intervals $t\in[t_{0}^{q},\widehat{t}_{0}^{q}]$, $q\in\mathbb{N}$ and this completes the proof.
\end{proof}

\section{Special Cases and Convergence Details}

\label{specialcases}

Firstly, given Theorem 1, it is worth noting the following result.

\begin{proposition} \label{convergedMatrixProp}
Consider a group of agents $\mathcal{V}$ and network $\mathcal{G}(t)(\mathcal{V},\mathcal{E}(t))$ where each agent $i$ has initial state $\mu_{i}(0)\in\mathfrak{U}_{p}\left(\mathbb{R}\right)$ and updates its state $\mu_{i}(t)$ according to (\ref{consensus3}). Suppose for all $t_{0}\in\mathbb{N}$ the graph union $\mathfrak{G}(t_{0},\infty)$ is connected so that Theorem 1 applies and there exists $\mu^{\ast}\in\mathfrak{U}_{p}(\mathbb{R})$ such that $\lim_{t\rightarrow\infty}\ell_{p}(\mu_{i}\left(t\right),\mu^{\ast})=0$ for any $i\in\mathcal{V}$. Then there exists some symmetric weight matrix $\overline{\mathbf{W}}=[\overline{w}_{ij}]\in\mathbb{R}^{n\times n}$ with $\overline{w}_{ij}\in(0,1)$ and $\sum_{j\in\mathcal{V}}\overline{w}_{ij}=1$ for all $i$ such that \ for any $i\in\mathcal{V}$
\[
	\mu^{\ast}=\argmin_{\eta\in\mathfrak{U}_{p}(\mathbb{R})}\sum_{j\in\mathcal{V}}\overline{w}_{ij}\,\ell_{p}(\eta,\mu_{j}(0))^{p}
\]
where we emphasize that $\overline{\mathbf{W}}$ is not (generally) the same as ${\mathbf{W}}(t)$ but it is solely dependent on the sequence ${\mathbf{W}}(t)$, $t\in\mathbb{N}$ and (possibly) the initial measures $\left\{ \mu_{i}(0)\right\} _{i\in\mathcal{V}}$.
\end{proposition}

Proof of this proposition is straightforward given the actual convergence result stated in Theorem 1. This result states that the common measure which all agent states converge to is within the convex hull of all initial agent measures in $\mathfrak{U}_{p}$.

An interesting open problem is how one can design the evolution of ${\mathbf{W}}(t)$, $t\in\mathbb{N}$ such that for a set of measures $\left\{ \mu_{i}(0)\right\} _{i\in\mathcal{V}}$ the final weighting matrix $\overline{\mathbf{W}}$ specifies a limit $\mu^{\ast}$, i.e. $\lim_{t\rightarrow\infty}\ell_{p}(\mu_{i}\left(t\right),\mu^{\ast})=0$ for any $i\in\mathcal{V}$ , that is optimal, or desired, in some sense (e.g. minimum variance over all possible $\overline{\mathbf{W}}$ given $\mu_{i}(0)\in\mathfrak{U}_{p}\left(\mathbb{R}\right)$, $i\in\mathcal{V}$). One example of ``average'' consensus is given later.

In the remainder of this section we consider convergence to particular limits of interest, e.g. to a limit equally close to all agent's initial measures. We also consider convergence speeds and we consider computational aspects of the update protocol for given classes of input measures.

\subsection{General Convergence Speeds and Average Wasserstein Consensus}

In this subsection we consider only undirected, connected, and time-invariant network graphs $\mathcal{G}(\mathcal{V},\mathcal{E})$. We consider only measures with finite second moment and we work solely in the Wasserstein metric space denoted by $(\mathfrak{U}_{2}(\mathbb{R}),\ell_{2})$.

The first result considers the convergence speed of the entire group of agents under the protocol (\ref{consensus3}).

\begin{proposition} \label{convergedSpeedProp}
Consider a group of agents $\mathcal{V}$ and a connected time-invariant network $\mathcal{G}(\mathcal{V},\mathcal{E})$ where each agent $i$ updates its state $\mu_{i}(t)\in\mathfrak{U}_{2}(\mathbb{R})$ according to (\ref{consensus3}). Then there exists $\mu^{\ast}\in\mathfrak{U}_{2}(\mathbb{R})$ such that for any $i\in\mathcal{V}$,
\[
	\lim_{t\rightarrow\infty}\ell_{2}(\mu_{i}(t),\mu^{\ast})=0
\]
at an exponential rate.
\end{proposition}

\begin{proof}
If $\mu_{i}(0)\in\mathfrak{U}_{2}(\mathbb{R})$ then the solution to (\ref{consensus3}) at any $i\in\mathcal{V}$ and any $t\in\mathbb{N}$ can be written in the form
\[
	\mu_{i}(t+1)(M)=\left({\textstyle \sum_{j\in\mathcal{N}_{i}}w_{ij}\, T_{j}^{i}(t)}\right)\#\,\mu_{i}(t)(M)
\]
for all Borel sets $M$ on $(\mathbb{R},d)$; see \cite{agueh2011barycenters}. Here $(T_{j}^{i}(t))\#\mu_{i}(t)$ denotes the push forward of $\mu_{i}(t)$ to $\mu_{j}(t)$ through the non-decreasing measurable map $T_{j}^{i}(t):\mathbb{R}\rightarrow\mathbb{R}$ such that $(T_{j}^{i}(t))\#\mu_{i}(t)=\mu_{j}(t)$. Obviously, we have $(T_{i}^{i}(t))\#\mu_{i}(t)=\mu_{i}(t)$. For any measure $\psi(t)$ dominated by the Lebesgue measure on $\mathbb{R}$ it follows that
\[
	\left({\textstyle \sum_{j\in\mathcal{N}_{i}}w_{ij}\, T_{j}^{i}(t)}\right)\#\,\mu_{i}(t)(M)=\left({\textstyle \sum_{j\in\mathcal{N}_{i}}w_{ij}\, T_{j}(t)}\right)\#\,\psi(t)(M)
\]
where $(T_{j}(t))\#\psi(t)$ denotes the push forward of $\psi(t)$ to $\mu_{j}(t)$; see \cite{bonneel2013sliced}. Write the cumulative distribution function for each $\mu_{i}(t)\in\mathfrak{U}_{2}(\mathbb{R})$ by $F_{i}(x):\mathbb{R}\rightarrow[0,1]$ and $F_{i}(x)=\mu_{i}\left((-\infty,x]\right)$. Define its inverse $F_{i}^{-}(x):[0,1]\rightarrow\mathbb{R}$ by
\[
	F_{i}^{-}(x)=\inf_{y}\left\{ y\in\mathbb{R}:F_{i}(y)\geq x\right\}
\]
for all $x\in[0,1]$. One can show \cite{agueh2011barycenters} that $T_{j}^{i}(t)=F_{j}^{-}\circ F_{i}$ or if $\psi(t)$ is uniform on $[0,1]$ then $T_{j}(t)=F_{j}^{-}$ and
\[
	\mu_{i}(t+1)(M)=\left({\textstyle \sum_{j\in\mathcal{N}_{i}}w_{ij}\, F_{j}^{-}}\right)\#\,\psi(t)(M)
\]
with $\psi(t)$ uniform on $[0,1]$. It follows directly that the solution to (\ref{consensus3}) at any $i\in\mathcal{V}$ and any $t\in\mathbb{N}$ has an inverse cumulative distribution function given by
\begin{equation}
	F_{i}^{-}(t+1)(x)={\textstyle \sum_{j\in\mathcal{N}_{i}}w_{ij}F_{j}^{-}(t)(x)} \label{invcumupdate}
\end{equation}
for all $x\in[0,1]$.

Now one can stack these functions so $F^{-}(t+1)(\mathbf{x})=\mathbf{W}F^{-}(t)(\mathbf{x})$ for all $\mathbf{x}\in[0,1]^{n}$. From the assumed network connectivity condition and the weighting assumptions we conclude that $\mathbf{W}$ is row-stochastic and primitive with a distinct maximum eigenvalue of $1$. The remaining $n-1$ eigenvalues have an absolute value strictly less than $1$. The convergence rate of $F^{-}(t)(\mathbf{x})$ is determined by the convergence rate of $\mathbf{W}^{t}$ to the rank one matrix $\mathbf{1}\mathbf{u}^{\top}$ associated with the maximum eigenvalue. Writing
\[
	\mathbf{W}^{t}=\sum_{i=1}^{n}\lambda_{i}^{t}\mathbf{v}_{i}\mathbf{u}_{i}^{\top}=\mathbf{1}\mathbf{u}^{\top}+\sum_{i=2}^{n}\lambda_{i}^{t}\mathbf{v}_{i}\mathbf{u}_{i}^{\top}
\]
where $\lambda_{i}$ is the $i$'th eigenvalue of $\mathbf{W}$, it then follows that $\|\mathbf{W}^{t}-\mathbf{1}\mathbf{u}^{\top}\|=\|\sum_{i=2}^{n}\lambda_{i}^{t}\mathbf{v}_{i}\mathbf{u}_{i}^{\top}\|$ vanishes exponentially at a rate dominated by the absolute value of the second largest eigenvalue (which is strictly less than 1) and the proof is complete.
\end{proof}

Note that a time-invariant network model is certainly not necessary for exponential convergence but we do not consider further generalisation in this work. It is also important to note that the time-varying network connectivity condition allowed in Theorem 1 is also certainly too weak to ensure exponential convergence in general. Indeed, Theorem 1 does not even require the network to be jointly connected until some arbitrary finite future time.

\begin{corollary} Consider a group of agents $\mathcal{V}$ and a connected time-invariant network $\mathcal{G}(\mathcal{V},\mathcal{E})$. Suppose that $\mathbf{W}$ is doubly stochastic and that each agent $i$ updates its state $\mu_{i}(t)\in\mathfrak{U}_{2}(\mathbb{R})$ according to (\ref{consensus3}). For any $i\in\mathcal{V}$, we have $\lim_{t\rightarrow\infty}\ell_{2}(\mu_{i}(t),\mu^{\ast})=0$ at an exponential rate where
\[
	\mu^{\ast}~=~\argmin_{\eta\in\mathfrak{U}_{2}(\mathbb{R})} ~\frac{1}{n}~{\textstyle \sum_{j\in\mathcal{V}}\ell_{2}(\eta,\mu_{j}(0))^{2}.}
\]
\end{corollary}

\begin{proof}
This result follows again because linear consensus in $\mathbb{R}$ over a time-invariant network with a doubly stochastic weighting matrix leads asymptotically to `average' consensus \cite{xiao2005scheme}. Looking at (\ref{invcumupdate}) we see that (nonlinear) consensus via (\ref{consensus3}) is related to (linear) consensus in the space of inverse cumulative distribution functions. Moving from the limiting inverse cumulative distribution function to a probability measure does not change the limiting $1/n$ averaging coefficient.
\end{proof}

This corollary provides sufficient conditions\footnote{A time-invariant network topology and a doubly stochastic weighting matrix. The time-invariance constraint can be relaxed (it is just sufficient) but we do not consider generalisation here.} under which exponential convergence to a measure is achieved and where the consensus measure achieved asymptotically at each agent is an average distance to all initial measures. In this case, as per Proposition \ref{convergedMatrixProp}, we have $\overline{\mathbf{W}} = \tfrac{1}{n}\mathbf{1}\mathbf{1}^{\top}$. We note that other consensus measures may be more desirable, e.g. one may want to reach an agreement on that measure with the smallest variance within the convex hull of all initial measures.

\subsection{Convergence with Gaussian Measures}

In this subsection we consider the flow of operation (\ref{consensus3}) when $\mu_{i}(0)$ is a Gaussian
measure. As in the preceding subsection, we consider only the case $(\mathfrak{U}_{2}(\mathbb{R}),\ell_{2})$.

Suppose $\mu_{i}(t)\in\mathfrak{U}_{2}(\mathbb{R})$ for all $i\in\mathcal{V}$ admits a Gaussian density of the form $\mathcal{N}({{p}}_{i},{{P}}_{i})$. Then it follows that $\mu_{i}(t+1)\in\mathfrak{U}_{2}(\mathbb{R})$ is a Gaussian measure \cite{mccann1997convexity,agueh2011barycenters} of density $\mathcal{N}(q,Q)$ where the updated mean and variance is given by
\begin{align*}
	q & = {\textstyle \sum_{j\in\mathcal{N}_{i}(t)}w_{ij}(t)\,p_{j}} \\
	Q & =\left({\textstyle \sum_{j\in\mathcal{N}_{i}(t)}w_{ij}(t)\,P_{j}^{1/2}}\right)^{2}
\end{align*}
One may find this result in the scalar case by studying (\ref{invcumupdate}) and noting in this case that,
\[
	F_{i}^{-}(t)(x) \,=\, p_i + P_i^{1/2}\,\sqrt{2}\,\mathrm{erf}^{-1}(2x-1)
\]
where $\mathrm{erf}(\cdot)$ is the standard error function. The (weighted) averaging in (\ref{invcumupdate}) taken point wise in $x$ is then just an averaging over the input means and standard deviations.

The following corollary then specialises those results in the preceding subsection to Gaussian measures.

\begin{corollary}
Consider a group of agents $\mathcal{V}$ and a connected time-invariant network $\mathcal{G}(\mathcal{V},\mathcal{E})$. Assume $\mathbf{W}$ is doubly stochastic and $\mu_{i}(0)\in\mathfrak{U}_{2}(\mathbb{R})$ admits a Gaussian density $\mathcal{N}(p_{i}(0),P_{i}(0))$. Then $\mu_{i}(t+1)\in\mathfrak{U}_{2}(\mathbb{R})$ in (\ref{consensus3}) admits a Gaussian density $\mathcal{N}(p_{i}(t+1),P_{i}(t+1))$ where
\begin{align*}
	p_{i}(t+1) & ={\textstyle \sum_{j\in\mathcal{N}_{i}}w_{ij}p_{j}(t),}\\
	P_{i}^{1/2}(t+1) & ={\textstyle \sum_{j\in\mathcal{N}_{i}}w_{ij}P_{j}^{1/2}(t).}
\end{align*}
Moreover we have for any $i\in\mathcal{V}$ $\lim_{t\rightarrow\infty}\ell_{2}(\mu_{i}(t),\mu^{\ast})=0$ exponentially fast where $\mu^{\ast}$ satisfies
\[
	\mu^{\ast}~=~\argmin_{\eta\in\mathfrak{U}_{2}(\mathbb{R})}~\frac{1}{n}~{\textstyle \sum_{j\in\mathcal{V}}\ell_{2}(\eta,\mu_{j}(0))^{2}.}
\]
\end{corollary}

This corollary collapses to a (classical) scalar average consensus algorithm \cite{saber:07} on the mean and standard deviation at each iteration. More general results achieving average consensus in this Gaussian setting that accommodate time-varying networks, finite-time convergence, etc. \cite{saber:04,ren2005consensus,saber:07,cortes2006finite} may be substituted.

Although the update (\ref{consensus3}) is linear (in mean and standard deviation) and closed in the event of Gaussian input measures, this is not generally true. The consensus problem (\ref{consensus3}) is, in general, inherently nonlinear.

\subsection{Computational Aspects of the Update Protocol}
In the case of Gaussian input measures, we have shown that the updating step of our consensus algorithm can be performed in closed form and that the resulting algorithm resembles a particular case of standard linear consensus in $\mathbb{R}$, e.g. see \cite{xiao2005scheme}.

Consider the important scenario where all the input initial measures are empirical measures\footnote{Interestingly, if each input measure is defined by a single Dirac (in $(\mathfrak{U}_{2}(\mathbb{R}),\ell_{2})$, i.e. with $p=2$), then the classical (linear) consensus algorithm in $\mathbb{R}$ is recovered, e.g. as in \cite{xiao2005scheme,moreau:05}. Of course, typically one is interested in more general empirical input measures.},
\[
	\mu_{i}^{N}(0)\left(dx\right)\,=\, \frac{1}{N}\,\sum_{j=1}^{N} \delta_{x_{j}^{i}(0)}\left(dx\right),
\]
where $\delta_{y}\left(dx\right)$ denotes the delta-Dirac measure located at $y$. In this case, the minimisation in (\ref{consensus3}) can be solved exactly. First we define the order statistics,
\[
	{x}_{1,*}^{i}(t) \,\leq\,{x}_{2,*}^{i}(t)\,\leq\, \ldots \,\leq\,{x}_{N,*}^{i}(t)
\]
corresponding to $\{x_{j}^{i}(t)\}$, $1\leq j\leq N$. Then with notation as in (\ref{consensus3}) we define,
\[
	\overline{x}^{\,i}_{k,*}(t)\,=\,\sum_{j\in\mathcal{N}_{i}(t)}w_{ij}(t)\,{x}_{k,*}^{\,j}(t)
\]
for all $1\leq k\leq N$. Then,
\[
	\mu_{i}^{N}(t+1)\left(dx\right)\,=\, \frac{1}{N}\,\sum_{j=1}^{N} \delta_{\overline{x}^{\,i}_{j,*}(t)}\left(dx\right),
\]
This computation involves only sorting and averaging of numbers in $\mathbb{R}$.

We note as an aside that for empirical measures in $\mathbb{R}^n$, for any integer $n\geq 1$, the minimization in (\ref{consensus3}) can still be solved exactly via a finite-dimensional linear program \cite{gangbo1998optimal,rabin2012wasserstein} and the resulting measure is again an empirical distribution\footnote{As discussed previously, for special, certainly non-generic, arrangements of discrete measures the minimisation in (\ref{consensus3}) may not have a unique solution (though a solution always exists) \cite{bonneel2015sliced}.}. However, the computational requirements of this linear program (in $\mathbb{R}^n$, when $n\geq2$) may explode quickly with the number of input measures and the number of atoms of these measures; see \cite{gangbo1998optimal,rabin2012wasserstein,bishop2014gossip,bishop2014information}. Numerous fast approximation methods have been derived for computing the barycenter update (\ref{consensus3}) itself \cite{boissard2011distribution,bonneel2015sliced,carlier2014numerical,cuturi2014fast}. The details of these algorithms are beyond the scope of this work, but it follows that the update (\ref{consensus3}) itself is thus computable with empirical measures in higher dimensions $\mathbb{R}^n$. Albeit not the focus of this article, we note in passing that convergence of the distributed protocol (\ref{consensus3}) with empirical measures on $\mathbb{R}^n$ follows readily from arguments on the convergence of Euclidean (linear) consensus \cite{moreau:05,bishop2014gossip}.

Consider now, more generally, arbitrary input measures on $\mathbb{R}$. The optimization problem (\ref{consensus3}) typically does not admit a closed form solution. However, it is convex \cite{agueh2011barycenters,bonneel2015sliced} and thus numerical methods/approximations are feasible and already exist in a number of cases; see \cite{bigot2012consistent,cuturi2014fast,bonneel2015sliced}. The issue of convergence under approximate update steps is unclear and beyond the scope of this work; but practically one would intuit that close approximation leads to close convergence. While convexity of the minimisation problem is advantageous in general, for measures on $\mathbb{R}$ there are yet further virtues. The update in (\ref{invcumupdate}) is typically computable in closed-form and thus we `almost' have a general closed-form expression for (\ref{consensus3}). The relationship between (\ref{consensus3}) and the inverse cumulative distribution in (\ref{invcumupdate}) is the basis for the solution with empirical measures on $\mathbb{R}$ given in the preceding paragraph. This relationship has also been explored in \cite{agueh2011barycenters,bigot2012consistent,bonneel2015sliced} with further example computations, and as a lead into more general computational results.

\section{Discussion and Applications}
\label{applications}

The output of each iteration of operation (\ref{consensus3}) is known in the literature as the Wasserstein barycenter. Similarly, the limit $\mu^{*}\in\mathfrak{U}_{p}(\mathbb{R})$ to which all agents converge upon repeated iteration of operation (\ref{consensus3}) is also a Wasserstein barycenter (on a fully connected graph of all agent's initial measures). In other words, this work studies the convergence properties of a consensus algorithm concerned with distributed (iterative) computation of the Wasserstein barycenter over a (possibly) time-varying, arbitrary, network topology. We consider undirected networks in this work for simplicity; but directed networks may be studied as in, e.g., \cite{moreau:05,ren2005consensus}, with additional conditions needed for convergence in that setting \cite{moreau:05}.

While this is the first such study in this direction, potential applications/uses for the Wasserstein barycenter (itself) have been considered previously
in a number of fields \cite{rabin2012wasserstein,bonneel2015sliced,carlier2014numerical,cuturi2014fast,srivastavawasp,bishop2014gossip,bishop2014information,arroyo2011smoothing,carlier2014numerical,boissard2011distribution} and this list is by no means exhaustive.

Arguably the most popular domain in which the Wasserstein barycenter has found applications is in computer vision and image/video processing
\cite{rabin2012wasserstein,bonneel2015sliced,carlier2014numerical}. We do not consider specifics here but the interested reader may consult
\cite{bonneel2015sliced} where numerous examples and a detailed discussion is given. It is noted \cite{bonneel2015sliced} that state-of-the-art advancements in a number of related problems have arisen via the use of Wasserstein barycenters. Importantly, both Gaussian and discrete measures find applicability through the Wasserstein barycenter in computer vision and image/video processing; again see \cite{bonneel2015sliced}.

Applications in machine learning and Bayesian statistics have also made use of the Wasserstein barycenter \cite{bigot2012consistent,cuturi2014fast,srivastavawasp} and it is envisioned that this technology (and the related optimal transportation problem) will find wider adoption in this field. In this setting, distributed (or even parallel) computation of the Wasserstein barycenter is likely important; e.g. distributed Bayesian computation on large data sets is the subject of \cite{srivastavawasp}.

Related work on consensus in spaces of probability distributions has been considered in the field of distributed estimation and information fusion. Suppose each agent starts with a (posterior) probability measure associated to some common underlying event of interest. Then one may like to combine all these measures (which amount to each agents estimate and/or belief of the underlying event) into a common probability measure that captures all the agents beliefs; this is often called opinion pooling. Related work in \cite{saber2005belief} considers the application of consensus to the problem of distributed Bayesian computations. In \cite{battistelli2014kullback,manuel2014distributed,bandyopadhyay2018distributed} the consensus algorithm from \cite{saber2005belief} is further extended and applied in distributed estimation and filtering. The Bayesian ideas in \cite{saber2005belief,battistelli2014kullback,manuel2014distributed,bandyopadhyay2018distributed} are related to so-called log-linear opinion pools which are related to the barycenter defined with respect to a Kullback-Leibler divergence (in an analogous fashion to the Wasserstein barycenter) \cite{bandyopadhyay2018distributed}. A Monte Carlo approximation of the consensus algorithm from \cite{saber2005belief} was studied in \cite{manuel2014distributed}. The log linear opinion pool for distributed information fusion was extended in \cite{taylor2019homogeneous,taylor2019consensus} to the general barycenter of a Bregman divergence (of which the Kullback-Leibler-based barycenter is a special case). Monte Carlo approximations were also considered in \cite{taylor2019homogeneous,taylor2019consensus}. While Bregman and Kullback-Leibler divergences have Bayesian-type interpretations, it is also possible to consider information fusion and distributed estimation in the context of Wasserstein barycenters as in \cite{bigot2012consistent,bishop2014gossip,bishop2014information}. As shown herein, and in \cite{bishop2014information}, information fusion with the Wasserstein barycenter has advantageous computability properties in the space of empirical distributions. Further study of the Wasserstein barycenter in the context of information fusion and estimation is an ongoing topic. 

We highlight finally that the ``consensus'' terminology throughout this article is used in the sense of network consensus and agreement as in \cite{degroot1974reaching,tsitsiklis:86,jadbabaie:03,saber:04,moreau:05,ren2005consensus,saber:07}. This differs from the topic of consensus (or ensemble) clustering or consensus aggregation \cite{vega2011survey,ghosh2015survey} which may also make use of the Wasserstein barycenter \cite{alvarez2018wide,verdinelli2019hybrid} for distributional clustering, etc. The latter topic may be an application of networked-type consensus as considered herein.

\section{Concluding Remarks}
\label{conclusion}

Distributed consensus in the Wasserstein metric space of probability measures was introduced in this paper. It is shown that convergence of the individual agents' measures to a common measure value is guaranteed if a relatively weak network connectivity condition is satisfied. The measure that is achieved asymptotically at each agent is the measure that minimises a weighted sum of its Wasserstein distances to these initial measures and is known as the Wasserstein barycenter in the literature.

Finally, we note that following \cite{moreau:05}, it would be straightforward to consider an extension to the case in which the network topology is directed and one expects analogous results (concerning connectivity) to apply in the Wasserstein space considered herein. For brevity, and notational simplicity, we do not explore this scenario further. Moreover, one may seek analogous results in the Wasserstein metric space of measures defined on the Borel sets of $(\mathbb{R}^{m},d)$ for some $m\geq2$. We conjecture that similar results hold in this case. However while many of the lemmas used herein carry over immediately, this generalization is not immediate. Indeed, the Wasserstein metric space in such cases is positively curved, so it does not resemble Euclidean space and it is not CAT(0).

{\footnotesize{}{}

 }

\begin{thebibliography}{10}

\bibitem{degroot1974reaching}
M.H. DeGroot.
\newblock Reaching a consensus.
\newblock {\em Journal of the American Statistical Association},
  69(345):118--121, March 1974.

\bibitem{tsitsiklis:86}
J.N. Tsitsiklis, D.P. Bertsekas, and M.~Athans.
\newblock Distributed asynchronous deterministic and stochastic gradient
  optimization algorithms.
\newblock {\em IEEE Trans. on Automatic Control}, 31(9):803--812, September
  1986.

\bibitem{jadbabaie:03}
A.~Jadbabaie, J.~Lin, and A.S. Morse.
\newblock Coordination of groups of mobile autonomous agents using nearest
  neighbor rules.
\newblock {\em IEEE Transactions on Automatic Control}, 48(6):988--1001, June
  2003.

\bibitem{saber:04}
R.~Olfati-Saber and R.M. Murray.
\newblock Consensus problems in networks of agents with switching topology and
  time-delays.
\newblock {\em IEEE Transactions on Automatic Control}, 49(9):1520--1533,
  September 2004.

\bibitem{moreau:05}
L.~Moreau.
\newblock Stability of multiagent systems with time-dependent communication
  links.
\newblock {\em IEEE Transactions on Automatic Control}, 50(2):169--182,
  February 2005.

\bibitem{ren2005consensus}
W.~Ren and R.W. Beard.
\newblock Consensus seeking in multiagent systems under dynamically changing
  interaction topologies.
\newblock {\em IEEE Transactions on Automatic Control}, 50(5):655--661, 2005.

\bibitem{saber:07}
R.~Olfati-Saber, J.A. Fax, and R.M. Murray.
\newblock Consensus and cooperation in networked multi-agent systems.
\newblock {\em Proceedings of the IEEE}, 95(1):215--223, January 2007.

\bibitem{cao2008reaching}
M.~Cao, A.S. Morse, and B.D.O. Anderson.
\newblock Reaching a consensus in a dynamically changing environment: {A}
  graphical approach.
\newblock {\em SIAM Journal on Control and Optimization}, 47(2):575--600, 2008.

\bibitem{xiao2005scheme}
L.~Xiao, S.~Boyd, and S.~Lall.
\newblock A scheme for robust distributed sensor fusion based on average
  consensus.
\newblock In {\em Proc. of the 4th International Symposium Information
  Processing in Sensor Networks}, pages 63--70, Los Angeles, California, 2005.

\bibitem{spanos2005distributed}
D.P. Spanos and R.M. Murray.
\newblock Distributed sensor fusion using dynamic consensus.
\newblock In {\em Proc. of the 16th IFAC World Congress}, Prague, Czech
  Republic, July 2005.

\bibitem{olfati2007distributed}
R.~Olfati-Saber.
\newblock Distributed {Kalman} filtering for sensor networks.
\newblock In {\em Proc. of the 46th IEEE Conference on Decision and Control},
  pages 5492--5498, New Orleans, Louisiana, December 2007.

\bibitem{carli2008distributed}
R.~Carli, A.~Chiuso, L.~Schenato, and S.~Zampieri.
\newblock Distributed {Kalman} filtering based on consensus strategies.
\newblock {\em IEEE Journal on Selected Areas in Communications},
  26(4):622--633, 2008.

\bibitem{cattivelli2010diffusion}
F.S. Cattivelli and A.H. Sayed.
\newblock Diffusion strategies for distributed {Kalman} filtering and
  smoothing.
\newblock {\em IEEE Transactions on Automatic Control}, 55(9):2069--2084, 2010.

\bibitem{hui2008distributed}
Q.~Hui and W.M. Haddad.
\newblock Distributed nonlinear control algorithms for network consensus.
\newblock {\em Automatica}, 44(9):2375--2381, 2008.

\bibitem{yu2011consensus}
W.~Yu, G.~Chen, and M.~Cao.
\newblock Consensus in directed networks of agents with nonlinear dynamics.
\newblock {\em IEEE Transactions on Automatic Control}, 56(6):1436--1441, 2011.

\bibitem{ajorlou2011sufficient}
A.~Ajorlou, A.~Momeni, and A.G. Aghdam.
\newblock Sufficient conditions for the convergence of a class of nonlinear
  distributed consensus algorithms.
\newblock {\em Automatica}, 47(3):625--629, 2011.

\bibitem{cortes2008distributed}
J.~Cort{\'e}s.
\newblock Distributed algorithms for reaching consensus on general functions.
\newblock {\em Automatica}, 44(3):726--737, 2008.

\bibitem{wang2010distributed}
X.~Wang and Y.~Hong.
\newblock Distributed finite-time $\chi$-consensus algorithms for multi-agent
  systems with variable coupling topology.
\newblock {\em Journal of Systems Science and Complexity}, 23(2):209--218,
  2010.

\bibitem{cortes2006finite}
J.~Cort{\'e}s.
\newblock Finite-time convergent gradient flows with applications to network
  consensus.
\newblock {\em Automatica}, 42(11):1993--2000, 2006.

\bibitem{zhu2010discrete}
M.~Zhu and S.~Mart{\'\i}nez.
\newblock Discrete-time dynamic average consensus.
\newblock {\em Automatica}, 46(2):322--329, 2010.

\bibitem{hong2006tracking}
Y.~Hong, J.~Hu, and L.~Gao.
\newblock Tracking control for multi-agent consensus with an active leader and
  variable topology.
\newblock {\em Automatica}, 42(7):1177--1182, 2006.

\bibitem{strogatz2000kuramoto}
S.H. Strogatz.
\newblock {From Kuramoto to Crawford: Exploring the onset of synchronization in
  populations of coupled oscillators}.
\newblock {\em Physica D: Nonlinear Phenomena}, 143(1):1--20, 2000.

\bibitem{dorfler2014synchronization}
F.~D{\"o}rfler and F.~Bullo.
\newblock {Synchronization in Complex Networks of Phase Oscillators: A Survey}.
\newblock {\em Automatica}, 50(6):1539--1564, 2014.

\bibitem{li2010consensus}
Z.~Li, Z.~Duan, G.~Chen, and L.~Huang.
\newblock Consensus of multiagent systems and synchronization of complex
  networks: {A} unified viewpoint.
\newblock {\em IEEE Trans. on Circuits and Systems I}, 57(1):213--224, 2010.

\bibitem{tuna2007consensus}
S.~Emre~Tuna and R.~Sepulchre.
\newblock Consensus under general convexity.
\newblock In {\em Proc. of the 46th IEEE Conference on Decision and Control},
  pages 294--299, New Orleans, USA, 2007.

\bibitem{sarlette2009consensus}
A.~Sarlette and R.~Sepulchre.
\newblock Consensus optimization on manifolds.
\newblock {\em SIAM J. on Control and Optimization}, 48(1):56--76, 2009.

\bibitem{Baras2010}
I.~Matei and J.S. Baras.
\newblock The asymptotic consensus problem on convex metric spaces.
\newblock In {\em Proc. of the 2nd {IFAC} Workshop on Distributed Estimation
  and Control in Networked Systems}, Annecy, France, Sept. 2010.

\bibitem{sepulchre2011consensus}
R.~Sepulchre.
\newblock Consensus on nonlinear spaces.
\newblock {\em Annual Reviews in Control}, 35(1):56--64, 2011.

\bibitem{Grohs2012Wolfowitz}
P.~Grohs.
\newblock {Wolfowitz's Theorem and Convergence of Consensus Algorithms in
  Hadamard Spaces}.
\newblock Technical Report SAM Report 2012-27, ETH, Zurich, Switzerland, August
  2012.

\bibitem{bonnabel2013stochastic}
S.~Bonnabel.
\newblock Stochastic gradient descent on {Riemannian} manifolds.
\newblock {\em IEEE Transactions on Automatic Control}, 58(9):2217--2229,
  September 2013.

\bibitem{garin2010survey}
F.~Garin and L.~Schenato.
\newblock A survey on distributed estimation and control applications using
  linear consensus algorithms.
\newblock In {\em Networked Control Systems}, pages 75--107. Springer, 2010.

\bibitem{givens1984class}
C.R. Givens and R.M. Shortt.
\newblock A class of {Wasserstein} metrics for probability distributions.
\newblock {\em The Michigan Mathematical Journal}, 31(2):231--240, 1984.

\bibitem{bishopdoucetIFAC}
A.N. Bishop and A.~Doucet.
\newblock Distributed nonlinear consensus in the space of probability measures.
\newblock In {\em Proc. of the 19th IFAC World Congress}, Cape Town, South
  Africa, August 2014, (submitted: October 27, 2013).

\bibitem{cv:03a}
C.~Villani.
\newblock {\em Topics in Optimal Transportation}.
\newblock American Mathematical Society, 2003.

\bibitem{ambrosio2005gradient}
L.~Ambrosio, N.~Gigli, and G.~Savara\'e.
\newblock {\em Gradient Flows in Metric Spaces and in the Space of Probability
  Measures}.
\newblock Birkhauser Verlag, Basel, Switzerland, 2005.

\bibitem{Kloeckner2010}
B.~Kloeckner.
\newblock A geometric study of {Wasserstein} spaces: Euclidean spaces.
\newblock {\em Annali della Scuola Normale Superiore di Pisa}, IX(2):297--323,
  2010.

\bibitem{bertrand2012geometric}
J.~Bertrand and B.~Kloeckner.
\newblock {A Geometric Study of {Wasserstein} Spaces: Hadamard Spaces}.
\newblock {\em Journal of Topology and Analysis}, 4(4):515--542, December 2012.

\bibitem{mccann1997convexity}
R.J. McCann.
\newblock A convexity principle for interacting gases.
\newblock {\em Advances in Mathematics}, 128(1):153--179, 1997.

\bibitem{major1978invariance}
P.~Major.
\newblock On the invariance principle for sums of independent identically
  distributed random variables.
\newblock {\em Journal of Multivariate Analysis}, 8(4):487--517, 1978.

\bibitem{agueh2011barycenters}
M.~Agueh and G.~Carlier.
\newblock Barycenters in the {Wasserstein} space.
\newblock {\em SIAM Journal on Mathematical Analysis}, 43(2):904--924, 2011.

\bibitem{bigot2012consistent}
J.~Bigot and T.~Klein.
\newblock Consistent estimation of a population barycenter in the {Wasserstein}
  space.
\newblock {\em arXiv e-print arXiv:1212.2562}, 2012.

\bibitem{bonneel2015sliced}
N.~Bonneel, J.~Rabin, G.~Peyr{\'e}, and H.~Pfister.
\newblock Sliced and {Radon Wasserstein} barycenters of measures.
\newblock {\em Journal of Mathematical Imaging and Vision}, 51(1):22--45, 2015.

\bibitem{burago2001course}
D.~Burago, Y.~Burago, and S.~Ivanov.
\newblock {\em A Course in Metric Geometry}.
\newblock American Mathematical Society, Providence, R.I., 2001.

\bibitem{bridson1999metric}
M.R. Bridson and A.~Haefliger.
\newblock {\em Metric Spaces of Non-Positive Curvature}.
\newblock Springer, Berlin, Germany, 1999.

\bibitem{boissard2011distribution}
E.~Boissard, T.~Le~Gouic, and J.-M. Loubes.
\newblock Distribution's template estimate with {Wasserstein} metrics.
\newblock {\em arXiv e-print arXiv:1111.5927}, 2011.

\bibitem{bonneel2013sliced}
N.~Bonneel, J.~Rabin, G.~Peyr{\'e}, and H.~Pfister.
\newblock Sliced and {Radon Wasserstein} barycenters of measures.
\newblock {\em Journal of Mathematical Imaging and Vision}, pages 1--24, 2013.

\bibitem{gangbo1998optimal}
W.~Gangbo and A.~Swiech.
\newblock Optimal maps for the multidimensional {Monge-Kantorovich} problem.
\newblock {\em Communications on Pure and Applied Mathematics}, 51(1):23--45,
  January 1998.

\bibitem{rabin2012wasserstein}
J.~Rabin, G.~Peyr{\'e}, J.~Delon, and M.~Bernot.
\newblock Wasserstein barycenter and its application to texture mixing.
\newblock In {\em Scale Space and Variational Methods in Computer Vision},
  pages 435--446. Springer, 2012.

\bibitem{bishop2014gossip}
A.N. Bishop.
\newblock Gossip-based distributed data fusion of empirical probability
  measures.
\newblock In {\em Proc. of the 2014 IEEE Workshop on Statistical Signal
  Processing}, pages 372--375, Gold Coast, Australia, 2014.

\bibitem{bishop2014information}
A.N. Bishop.
\newblock Information fusion via the {Wasserstein} barycenter in the space of
  probability measures: {Direct fusion of empirical measures and Gaussian
  fusion with unknown correlation}.
\newblock In {\em Proc. of the 17th International Conference on Information
  Fusion}, Salamanca, Spain, July 2014, (submitted: February 18, 2014).

\bibitem{carlier2014numerical}
G.~Carlier, A.~Oberman, and E.~Oudet.
\newblock Numerical methods for matching for teams and {Wasserstein}
  barycenters.
\newblock {\em arXiv e-print arXiv:1411.3602}, 2014.

\bibitem{cuturi2014fast}
M.~Cuturi and A.~Doucet.
\newblock Fast computation of {Wasserstein} barycenters.
\newblock In {\em Proc. of the 31st International Conference on Machine
  Learning}, pages 685--693, Beijing, China, 2014.

\bibitem{srivastavawasp}
S.~Srivastava, V.~Cevher, Q.~Tran-Dinh, and D.B. Dunson.
\newblock {WASP: Scalable Bayes} via barycenters of subset posteriors.
\newblock In {\em Proc. of the 18th International Conference on Artificial
  Intelligence and Statistics}, pages 912--920, San Diego, USA, 2015.

\bibitem{arroyo2011smoothing}
J.~Arroyo, G.~Gonz{\'a}lez-Rivera, C.~Mat{\'e}, and A.M. San~Roque.
\newblock Smoothing methods for histogram-valued time series: {An} application
  to value-at-risk.
\newblock {\em Statistical Analysis and Data Mining: The ASA Data Science
  Journal}, 4(2):216--228, April 2011.

\bibitem{saber2005belief}
R.~Olfati-Saber, E.~Franco, E.~Frazzoli, and J.S. Shamma.
\newblock Belief consensus and distributed hypothesis testing in sensor
  networks.
\newblock In {\em Proc. of the Workshop on Network Embedded Sensing and
  Control}, Notre Dame University, South Bend, Indiana, October 2005.

\bibitem{battistelli2014kullback}
G.~Battistelli and L.~Chisci.
\newblock {Kullback--Leibler} average, consensus on probability densities, and
  distributed state estimation with guaranteed stability.
\newblock {\em Automatica}, 50(3):707--718, 2014.

\bibitem{manuel2014distributed}
I.L. Manuel and A.N. Bishop.
\newblock Distributed {Monte Carlo} information fusion and distributed particle
  filtering.
\newblock In {\em Proc. of the 19th IFAC World Congress}, pages 8681--8688,
  Cape Town, South Africa, August 2014.

\bibitem{bandyopadhyay2018distributed}
S.~Bandyopadhyay and S.-J. Chung.
\newblock Distributed {Bayesian} filtering using logarithmic opinion pool for
  dynamic sensor networks.
\newblock {\em Automatica}, 97(11):7--17, 2018.

\bibitem{taylor2019homogeneous}
C.N. Taylor and A.N. Bishop.
\newblock Homogeneous functionals and {Bayesian} data fusion with unknown
  correlation.
\newblock {\em Information Fusion}, 45:179--189, 2019.

\bibitem{taylor2019consensus}
C.N. Taylor and A.N. Bishop.
\newblock Distributed power mean fusion.
\newblock In {\em Proceedings 22nd International Conference on Information
  Fusion}, 2019.

\bibitem{vega2011survey}
S.~Vega-Pons and J.~Ruiz-Shulcloper.
\newblock A survey of clustering ensemble algorithms.
\newblock {\em International Journal of Pattern Recognition and Artificial
  Intelligence}, 25(03):337--372, 2011.

\bibitem{ghosh2015survey}
J.~Ghosh and A.~Acharya.
\newblock A survey of consensus clustering.
\newblock In {\em Handbook of Cluster Analysis}, pages 518--539. Chapman and
  Hall/CRC, 2015.

\bibitem{alvarez2018wide}
P.C. Alvarez-Esteban, E.~del Barrio, J.A. Cuesta-Albertos, and C.~Matr{\'a}n.
\newblock Wide consensus aggregation in the {Wasserstein} space: {Application}
  to location-scatter families.
\newblock {\em Bernoulli}, 24(4A):3147--3179, 2018.

\bibitem{verdinelli2019hybrid}
I.~Verdinelli and L.~Wasserman.
\newblock Hybrid {Wasserstein} distance and fast distribution clustering.
\newblock {\em Electronic Journal of Statistics}, 13(2):5088--5119, 2019.

\end{thebibliography}
\end{document}